\documentclass[10pt,reqno]{amsart}
\usepackage{amsmath}
\usepackage{amssymb}
\usepackage{amsfonts}
\usepackage{epsfig}
\usepackage{graphicx} 
\usepackage{latexsym,amssymb,epsf}

\newtheorem{theorem}{Theorem}[section]
\newtheorem{lemma}{Lemma}[section]

\newtheorem{remark}{Remark}[section]

\newtheorem{example}{Example}[section]
\numberwithin{equation}{section}
\begin{document}
\title{Minimal Cost of a Brownian Risk without Ruin}
\author{Shangzhen Luo}
\address{Department of Mathematics, University of Northern Iowa, Cedar Falls, Iowa 
USA 50614-0506}
\email{luos@uni.edu}
\author{Michael Taksar}
\address{Department of Mathematics, University of Missouri, Columbia, MO, USA 60211}
\email{taksar@math.missouri.edu}
\thanks{This work was supported by the Norwegian Research Council: Forskerprosjekt ES445026, ``Stochastic Dynamics of Financial Markets.'' 
The first author acknowledges a UNI summer research fellowship.}
\begin{abstract}
In this paper, we study a risk process modeled by a Brownian motion
with drift (the diffusion approximation model). 
The insurance entity can purchase reinsurance
to lower its risk and receive cash injections at discrete times to avoid ruin.
Proportional reinsurance and excess-of-loss reinsurance are considered.
The objective is to find the optimal reinsurance and cash injection strategy
that minimizes the total cost to keep the company's surplus process non-negative, 
i.e. without ruin, 
where the cost function is defined as the total discounted value of the injections.
The optimal solution is found explicitly 
by solving the according quasi-variational inequalities (QVIs).

\end{abstract}
\maketitle

\section{Introduction}
This paper studies a model based on management of a mutual insurance entity,
where operations of the company involve reinsurance purchase and cash injection calls.
That is, the company can manage its risk level 
by buying reinsurance and increase its surplus level by calling for
 cash injections.
The objective is to find the optimal reinsurance-injection strategy to
minimize the maintenance cost to keep the company functioning without ruin. The resulting problem
is an optimal stochastic control problem with both classical or regular 
(continuous)
and impulse (discrete) control components, where the surplus process is modeled by a
Brownian motion with drift, i.e. the diffusion approximation model.

Capital injection problems under the classical risk model
have been considered in Dickson and Waters \cite{DW} where capital injections must
be made whenever negative surplus level occurs. 
Eisenberg and Schmidli \cite{ES} studied capital injection minimization under both the diffusion approximation model 
and the classical risk model with a constant interest rate. 
There the optimal injection process was found to be 
the local time of the diffusion process
and the injection process was continuous in time.
By the Hamilton-Jacobi-Bellman (HJB) equation method, 
the optimal value function was found explicitly.
Note that related problems with controls of injections and 
withdrawals under a model of Brownian storage system
were firstly investigated in Harrison and Taylor \cite{HT}.

In this paper, we suppose cash injections are made discretely with a 
fixed set-up cost for each injection.
At the same time, the company can 
purchase reinsurance to lower its risk level.
Thus the controlled surplus process is a jump-diffusion process.
The objective is to find the optimal injection-reinsurance control policy
that minimizes the total discounted value of the injections (cost function) while 
keeping the surplus level non-negative. 
As a result, the optimal control problem is formulated as a mixed 
regular-impulse control problem.
 Quasi-variational inequalities are 
given to characterize the optimal value function. 
The QVIs are then solved explicitly for the minimal cost functions and 
the associated optimal reinsurance-injection controls are found.
For other applications of QVI methods and regular-impulse control theory, 
readers are referred to Bensoussan et al \cite{B} and 
Candenillas et al \cite{C1} and \cite{C2}.
In our paper we consider two types of reinsurance - proportional reinsurance 
and excess-of-loss reinsurance separately. We note that these two types
 in a certain sense represent two extremes of reinsurance 
based on the expected value principle and the variance principle
 (see \cite{TZ}). 

The paper is organized as follows. 
The optimization problem is formulated in Section 2.
Section 3 investigates the case with proportional reinsurance and Section 4 focuses on
excess-of-loss reinsurance. Some examples and concluding remarks are given in Section 5.

\section{The Mathematical Model}
We begin with the Cramer-Lundberg model of an insurance entity:
$$X_t=X_0 + pt-\underset{i=1}{\overset{N(t)}{\sum}}Y_i,$$
where $x$ is the initial surplus, 
 $p$ is the premium rate, $N(t)$ is a Poisson process with constant intensity $\lambda$, the random variables
 $Y_i$'s are positive i.i.d. claims with a common distribution function $F$. 
Without loss of generality, we can assume $\lambda=1$.
The process can be approximated by the following
 drifted Brownian motion, i.e. the diffusion approximation model (See e.g. \cite{EHT} and \cite{TM}):
 $$X_t=x + \int_0^t\mu ds+\int_0^t\sigma dw_s,$$
where parameters $\mu$ and $\sigma$ are as follows:
$\mu=p-E(Y_1)$ and $\sigma=\sqrt{\lambda E(Y_1^2)}$, and $\{w_s\}_{s\geq 0}$ is a standard
Brownian motion adapted to {\it information filtration} $\{\mathcal{F}_s\}_{s\geq 0}$ 
in a probability space $(\Omega,\mathcal{F}, P)$. 
Further suppose that the insurance entity has a {\it debt} or {\it dividend liability} 
which is funded from the surplus at a constant rate.
 Without any control, the surplus process
is then governed by:
$$X_t=x + \int_0^t(\mu-\delta) ds+\int_0^t\sigma dw_s,$$
where $\delta>0$ is the rate of debt liability.

Now we define a control policy $\pi$ given by a triple
$$\pi:=\{u_s, s>0; (\tau_1,\tau_2,...); (\xi_1,\xi_2,...)\},$$
where $0\leq u_s\leq 1$ is a predictable process with respect to $\mathcal{F}_s$,
the random variables $\tau_i$'s are an increasing sequence of 
stopping times with respect to  the
filtration $\{\mathcal{F}_s\}_{s>0}$ and $\xi_i$ is an $\mathcal{F}_{\tau_i}$ measurable random variable,
$i=1,2,...$. The stopping times $\tau_i$'s represent the times at which the
cash injections are called for;
and the random variables $\xi_i$ represent the sizes of the $i$th
cash injections.
The adaptedness of the control to the information filtration restricts that 
any decision can be made 
based on the past history but not on the future information.
The quantity $u_s$ is the level of {\it risk exposure} at time $s$.
For {\it proportional reinsurance}, $u_s$ represents the fraction of each claim
covered by the insurance entity while the other fraction $(1-u_s)$ of the claim
 is paid by the reinsurance company.
For {\it excess-of-loss} reinsurance, $u_s$ represents the maximum limit of 
the insurance payment while the exceeding loss $(Y_i-u_s)^+$ is paid by the reinsurer.
We note that the reinsurance purchase reduces 
both the risk level $\sigma$ and premium rate $\mu$.

Given any control $\pi$, the surplus process under the diffusion
approximation is described by the following dynamics:
\begin{equation}\label{X}
X^\pi_t=x+\int_0^t[\mu(u_s)-\delta]ds+\int_0^t\sigma(u_s) dw_s+\underset{\tau_i\leq t}{\sum}\xi_i,
\end{equation}
where
\begin{equation}\label{prein}
\mu(u)=\mu u,\ \ \sigma(u)=\sigma u,
\end{equation}
for $0\leq u\leq 1$ in the case of proportional reinsurance, and
\begin{equation}\label{exrein}
\mu(u)=\int_0^u[1-F(x)]dx,\ \sigma(u)=\sqrt{\int_0^u2x[1-F(x)]dx},
\end{equation}
for $0\leq u\leq N$ in the case of excess-of-loss reinsurance, here
$N\le\infty$ is the upper bound of the support of the claim size distribution $F$.
We shall write $\mu=\mu(N)$ and $\sigma^2=\sigma^2(N)$ (the first two moments of $F$),
and denote by $\mathcal{U}:=[0,1]$ or $\mathcal{U}:=[0,N]$ the reinsurance control
region for proportional reinsurance or excess-of-loss reinsurance.

Now we define the {\it cost function} under policy $\pi$ as the following:
\begin{equation}
C^\pi(x):=E_x[\underset{\tau_i<\infty}{\sum}e^{-r\tau_i}g(\xi_i)],
\end{equation}
where $r>0$ is the discounting rate, $E_x$ stands for expectation under $P(\cdot|X_0=x)$,
and $g$ is the cost function of cash injection calls defined by:
\begin{equation}
g(\xi)=K+c\xi,
\end{equation}
here $K >0$ is the {\it fixed set-up cost} of a call 
and $c \geq 1$ is the {\it proportional cost rate} that represents the amount of cash needed to be raised
in order for one dollar to be added to the surplus process. We note that
if $c>1$, it represents that there is a positive proportional cost to raise cash in the surplus process.
A control policy $\pi$ is {\it admissible} 
if under policy $\pi$, the surplus level is always above zero, i.e.
$X^\pi_t\geq 0$,
almost surely for all $t$, and the cost function is finite, i.e.
$C^\pi(x)<\infty$,
for all $x\geq 0$. We denote by $\Pi$ the set of all admissible control policies.
The objective is to find the value function of the control problem,
i.e. the minimal cost function given by:
\begin{equation}
\label{V}
V(x):=\underset{\pi\in\Pi}{\inf} C^\pi(x),
\end{equation}
and the optimal control policy $\pi^*$ such that
\begin{equation*}
V(x)=C^{\pi^*}(x).
\end{equation*}
We note that if there is no debt liability ($\delta=0$), the optimal control
problem is trivial and $V(x)\equiv 0$. The company
buys 100$\%$ reinsurance and stays at a positive level indefinitely
without any cash injections.

\section{Properties of the value function and the QVIs}
In this section, we show a few properties on the value function $V$
 and give the QVIs that govern this function. 
 First we show existence of an admissible control and boundedness of $V$.
 \begin{lemma} 
 For any $\xi>0$, 
$ V(0)\leq g(\xi)\frac{e^{-r\frac{\xi}{\delta}}}{1-e^{-r\frac{\xi}{\delta}}}$,
and for any $x\geq 0$,
 $V(x)\leq V(0)e^{-r\frac{x}{\delta}}$.
 \end{lemma}
 To prove the first inequality, one considers
a control policy with $u_s\equiv 0$. Then the surplus process
becomes deterministic. Further assume constant injections of $\xi$ 
are made at the times of ruin; then the cost function under this control is given by the right side of
the inequality. The second inequality can be proved similarly.
As a consequence  we have the following boundary condition:
\begin{equation}\label{bd1}
V(\infty)=0.
\end{equation}
Obviously the value function $V$ is a deceasing function.
Next we show that the optimal cash injections occur at the times of ruin, i.e.
a cash injection call should be made only when the surplus level hits zero.
\begin{lemma}\label{l3}
If for a control policy $\pi$, there exists $i$ such that
$X^\pi_{\tau_i-}>0$, then the policy $\pi$  is not optimal.
\end{lemma}
Below we present the idea of the proof.
If it holds $X^\pi_{\tau_i-}>0$, then we consider a new control
policy that postpones the $i^{th}$ cash call until  the next ruin time, leaving 
all other functional and random variables the same as in the original policy.
Then the new policy reduces the total cost due to a longer period of discounting on 
the $i^{th}$ injection.
Thus cash injections should never be made at positive surplus levels in order to minimize the cost.

For a fixed $0\leq u\leq 1$, define an infinitesimal generator $\mathcal{L}^u$ as the following:
\begin{equation*}
(\mathcal{L}^u\phi)(x)=\frac{1}{2}\sigma^2(u)\frac{d^2\phi(x)}{dx^2}+[\mu(u)-\delta]\frac{d\phi(x)}{dx},
\end{equation*}
for any function $\phi$ in $C^2[0,\infty)$.
Define an {\it inf-convolution operator} as:
\begin{equation*}
M\phi(x)=\underset{\xi>0}{\inf}[g(\xi)+\phi(x+\xi)].
\end{equation*}
Suppose the value function $V$ is a $C^2[0,\infty)$ function, 
then the quasi-variational inequalities
of the control problem are given by:
\begin{equation*}
\begin{split}
\mathcal{L}^uV(x)-rV(x)&\geq 0,\\
MV(x)\geq V(x),
\end{split}
\end{equation*}
for all $x\geq 0$ and $u\in\mathcal{U}$, together with the tightness condition
$$[MV(x)-V(x)]\underset{u\in\mathcal{U}}{\min}[\mathcal{L}^uV(x)-rV(x)]=0.$$
For details of the derivation of the QVIs, we refer the reader to \cite{W}.
From boundary the condition~\eqref{bd1} and Lemma~\ref{l3}, we conjecture
that on $(0, \infty)$ the function $V$ solves:
\begin{equation}\label{E1}
\underset{u\in\mathcal{U}}{\min}[\mathcal{L}^uV(x)-rV(x)]=0,
\end{equation}
with the boundary conditions

\begin{equation}\label{bd}
V(\infty)=0,\ MV(0)=V(0).
\end{equation}

\section{Optimal Solution in the Case of Proportional Reinsurance}
In this section, we solve the QVIs and 
find an explicit expression of the value function 
$V \in C^2[0,\infty)$ in the case of proportional reinsurance. 
We also derive an optimal control policy $\pi^*$.
As we will see the qualitative nature of the control policy depends on the 
interrelation of the model parameters. Accordingly, two parameter cases will
be treated separately.

\subsection{The case of a low debt liability rate: $\delta<\frac{\mu^2+2r\sigma^2}{2\mu}$}
For any function $W$ in $C^2[0,\infty)$ with $W''(x)\neq 0$, write
$$u_W(x):=-\frac{\mu}{\sigma^2}\frac{W'(x)}{W''(x)}.$$
Suppose $V$ solves equation~\eqref{E1} at $x$, i.e.
\begin{equation}\label{E2}
\underset{0\leq u\leq 1}{\min}[\frac{1}{2}u^2\sigma^2V''(x)+(u\mu-\delta)V'(x)-rV(x)]=0,
\end{equation}
and suppose $V''(x)>0$ and $0\leq u_V(x)\leq 1$. 
Then the minimizer of the right side of \eqref{E2} is
$$u^*(x)=u_V(x).$$
Plugging this into \eqref{E2} and simplifying, we get
\begin{equation}\label{E3}
\frac{1}{2}\frac{\mu^2}{\sigma^2}\frac{[V'(x)]^2}{V''(x)}+\delta V'(x)+ rV(x)=0.
\end{equation}
Due to boundary condition $V(\infty)=0$, we conjecture a solution of the following form
$$V(x)=\alpha e^{-\beta x},$$
where $\alpha>0$ and $\beta>0$.
Plugging the quantities $V'(x)=-\beta V(x)$ and $V''(x)=\beta^2V(x)$ 
into \eqref{E3} and simplifying,
results in
\begin{equation}\label{beta}
\beta=\frac{1}{\delta}(r+\frac{1}{2}\frac{\mu^2}{\sigma^2}).
\end{equation}
Furthermore
\begin{equation}\label{u-star-1}
u_V(x)=\frac{\mu}{\sigma^2\beta}=\frac{2\delta\mu}{2r\sigma^2+\mu^2}.
\end{equation}
Thus $u_V(x)<1$ is equivalent to $\delta<\frac{\mu^2+2r\sigma^2}{2\mu}$, 
which is true due to the parameter assumption.
This verifies that the minimizer of the left side of equation~\eqref{E2} is 
$u^*(x)=u_V(x)$.
Next we determine $\alpha$ using the boundary condition $MV(0)=V(0)$.
From this condition follows
\begin{equation}
\label{a1}
\alpha=\underset{\xi>0}{\inf}[K+c\xi+\alpha e^{-\beta \xi}].
\end{equation}
By differentiation, the minimizer of the right side of \eqref{a1} is given by
\begin{equation}
\label{xi-star}
\xi^*=\frac{1}{\beta}\ln\frac{\alpha\beta}{c}.
\end{equation}
Thus $\alpha$ solves the following equation:
$$\alpha=K+c\xi^*+\alpha e^{-\beta \xi^*},$$
or
\begin{equation}
\label{alpha}
K+\frac{c}{\beta}\ln\frac{\alpha\beta}{c}+\frac{c}{\beta}-\alpha=0.
\end{equation}
For this solution to make sense, we must have $\xi^*>0$, which implies $\alpha>\frac{c}{\beta}$;
thus we need to find a solution $\alpha$ of equation \eqref{alpha}
over the interval $(c/\beta,\infty)$.
Write
\begin{equation}\label{G}
G(\alpha):=\frac{c}{\beta}\ln\frac{\alpha\beta}{c}+\frac{c}{\beta}-\alpha.
\end{equation}
Notice that $G(c/\beta)=0$ and $G(\infty)=-\infty$ and that
 the function is strictly decreasing  on $(c/\beta,\infty)$.
Therefore there always exists a unique solution of the equation $G(\alpha)=-K$ 
on interval $(c/\beta,\infty)$. 

Summarizing the discussions above and applying the verification Theorem~\ref{Ver}
we obtain 
\begin{theorem}\label{th1}
If $\delta<\frac{\mu^2+2r\sigma^2}{2\mu}$, then the value function defined 
in \eqref{V} is given by 
$$V(x)=\alpha e^{-\beta x},$$
where $\beta$ is given in \eqref{beta} and $\alpha$ is the unique solution
of equation \eqref{alpha} on the interval $(c/\beta,\infty)$.
The optimal reinsurance control is a constant given by
$$u^*_s\equiv \frac{2\delta\mu}{2r\sigma^2+\mu^2},$$
the optimal cash injections $\xi_i^*$, $i=1,2,...$ are a constant given by
$$\xi_i^*\equiv\frac{1}{\beta}\ln\frac{\alpha\beta}{c},$$
and the optimal times of injections $\tau_i^*$ 
are the ruin times of the surplus process under 
the policy $\pi^*$.
\end{theorem}
\begin{proof}
Note that $V''(x)>0$ and $$0<u_V(x)=\frac{2\delta\mu}{2r\sigma^2+\mu^2}<1$$ by the assumption.
Thus the minimizer of the quantity on the right side of \eqref{E1} over interval $[0,1]$, 
which is a quadratic function of $u$, is $u^*=u_V(x)$, i.e. it holds
$$\underset{0\leq u\leq 1}{\min}\mathcal{L}^uV(x)-rV(x)=\mathcal{L}^{u^*}V(x)-rV(x).$$ 
Simple calculation gives $$\mathcal{L}^{u^*}V(x)-rV(x)=0.$$ Thus $V$ solves equation \eqref{E1}.
Further the boundary condition $MV(0)=V(0)$ holds by the derivation of $\xi^*$ in \eqref{xi-star}.
The rest follows from the verification theorem.
\end{proof}

\subsection{The case with a high debt liability rate: $\delta\geq\frac{\mu^2+2r\sigma^2}{2\mu}$}
For any $x>0$, suppose $V$ solves equation~\eqref{E2}.
Further, $V''(x)>0$ and $u_V(x)\geq 1$. 
Then the minimizer of the right side of \eqref{E2} is
$$u^*(x)=1.$$
Plugging it in \eqref{E2}, we see that $V$ solves
\begin{equation}\label{E4}
\frac{1}{2}{\sigma^2}V''(x)+(\mu-\delta)V'(x)-rV(x)=0.
\end{equation}
A general solution of this equation is
$$V(x)=\eta e^{-\gamma x}+\kappa e^{\rho x},$$
where 
\begin{equation}\label{gamma}
\gamma, \rho=\frac{\sqrt{(\mu-\delta)^2+2r\delta}\pm(\mu-\delta)}{\sigma^2}
\end{equation}
and $\eta$ and $\kappa$ are free constants.
From the boundary condition $V(\infty)=0$, we get $\kappa=0$. Thus
\begin{equation}
V(x)=\eta e^{-\gamma x}.
\end{equation}
Since $V'(x)=-\gamma V(x)$ and $V''(x)=\gamma^2 V(x)$, we have
\begin{equation}\label{u-star-2}
u_V(x)=\frac{\mu}{\sigma^2\gamma}=\frac{\mu}{(\mu-\delta)+\sqrt{(\mu-\delta)^2+2r\sigma^2}}.
\end{equation}
Thus $u_V(x)\geq 1$ is equivalent to 
$\delta\geq\frac{\mu^2+2r\sigma^2}{2\mu}$, 
which holds by the parameter assumption.
Hence it follows that the minimizer of the left side of equation \eqref{E2} is $u^*(x)=1$.
Next we determine $\eta$ from boundary condition $MV(0)=V(0)$. This condition implies
\begin{equation}
\label{a2}
\eta=\underset{\xi>0}{\inf}[K+c\xi+\eta e^{-\gamma \xi}].
\end{equation}
By differentiation, the minimizer of the right side of \eqref{a2} is given by

\begin{equation}
\label{xi-star2}
\xi^*=\frac{1}{\gamma}\ln\frac{\eta\gamma}{c}.
\end{equation}
Thus $\eta$ solves:
$$\eta=K+c\xi^*+\eta e^{-\gamma \xi^*},$$
or
\begin{equation}
\label{eta}
K+\frac{c}{\gamma}\ln\frac{\eta\gamma}{c}+\frac{c}{\gamma}-\eta=0.
\end{equation}
Since $\xi^*$ in \eqref{xi-star2} must be positive, we should have $\eta>c/\gamma$.
In fact, it is easy to check that there always exists a unique 
solution $\eta$ of equation \eqref{eta} on the interval $(c/\gamma,\infty)$.

By virtue of the verification Theorem~\ref{Ver}, we have
\begin{theorem}\label{th2}
If $\delta\geq\frac{\mu^2+2r\sigma^2}{2\mu}$, then the value function defined 
in \eqref{V} is given by 
$$V(x)=\eta e^{-\gamma x},$$
where $\gamma$ is given in \eqref{gamma} and $\eta$ is the unique solution of equation \eqref{eta}
on the interval $(c/\gamma,\infty)$.
The optimal reinsurance control is
$$u^*_s\equiv 1.$$
The optimal cash injections $\xi_i^*$, $i=1,2,...$ are a constant given by
$$\xi_i^*\equiv\frac{1}{\gamma}\ln\frac{\eta\gamma}{c},$$
and the optimal times of injections $\tau_i^*$ are the ruin times of the surplus process under 
the policy $\pi^*$.
\end{theorem}
\begin{proof}
Note that $V''(x)>0$ and $$u_V(x)=\frac{\mu}{(\mu-\delta)+\sqrt{(\mu-\delta)^2+2r\sigma^2}}\geq 1,$$ 
due to our assumption.
Thus the minimizer of the quadratic quantity on the right side of \eqref{E1} in $u$ over the interval
 $[0, 1]$ is $u^*=1$, i.e. 
$$\underset{0\leq u\leq 1}{\min}\mathcal{L}^uV(x)-rV(x)=\mathcal{L}^{1}V(x)-rV(x).$$ 
Simple calculation yields
 $$\mathcal{L}^{1}V(x)-rV(x)=0.$$ Thus $V$ solves equation \eqref{E1}.
Further the boundary condition $MV(0)=V(0)$ holds by the definition of
  $\xi^*$ in \eqref{xi-star2}.
Hence the result follows from the verification theorem.
\end{proof}

Next we analyze briefly the interplay between 
the optimal policies and the model parameters.

\begin{remark} When  the debt liability rate $\delta$ is low, the optimal policy
always purchases reinsurance at a constant level. 
While the debt liability rate $\delta$ is high, 
the optimal level of risk exposure is 100\%, i.e. no reinsurance is bought.
\end{remark}

Below is a remark on relations between the optimal cash injection amount and
the fixed set-up cost $K$. 
\begin{remark} Function $G$ defined by \eqref{G} is a strictly decreasing function
on $(c/\beta,\infty)$, thus its inverse function $G^{-1}$ is a strictly decreasing function
on $(-\infty, 0)$. So the quantity $$\alpha=G^{-1}(-K)$$ increases and so does 
$\xi^*$ in \eqref{xi-star} when $K$ increases.
This implies that when the fixed set-up cost is more expensive with other model parameters unchanged,
then one should inject more at each cash call. This agrees with the intuition, since the
larger the set-up costs of $K$ is, the less frequent the injections should happen.
\end{remark}

Next we find relations between the 
optimal cash injection and the proportional cost rate $c$.
Write $\theta=\alpha/c$ and define:
\begin{eqnarray*}
H(\theta):=\frac{1}{\beta}\ln\beta\theta+\frac{1}{\beta}-\theta,
\end{eqnarray*}
for $\theta\in (1/\beta,\infty)$.

\begin{remark} Function $H(\theta)$ is 
a strictly decreasing function on $(1/\beta,\infty)$ and its inverse function
$H^{-1}$ is a decreasing function on $(0,-\infty)$.
Note that $G(\alpha)=cH(\theta)$ and $G(\alpha)=-K$, thus it holds 
$$\theta=H^{-1}(-K/c),$$
wherefrom we see that if $c$ increases, then $\theta$ decreases and 
$\xi^*=1/\beta\ln(\beta\theta)$ decreases. 
This implies that when the proportional cost rate $c$ of the injections increases,
the optimal injection amount should decrease.
\end{remark}

The remark below describes what happens if the initial surplus is negative.
\begin{remark}
If a negative initial surplus level is allowed (that is $x<0$), then from
$$V(x)=MV(x),$$ 
one can show that the optimal policy requires a cash injection 
of $(\xi^*-x)$ at time $0$ where $\xi^*$ is defined by \eqref{xi-star} or \eqref{xi-star2}
depending on  the parameters.
The minimal cost function in this case is given by
$$V(x)=K+c(\xi^*-x)+V(\xi^*),$$
for $x<0$. That is
$$V(x)=K+\frac{c}{\beta}\ln\frac{\alpha\beta}{c}+\frac{c}{\beta}-cx=\alpha-cx,$$
for $x<0$ when $\delta<\frac{\mu^2+2r\sigma^2}{2\mu}$, and
$$V(x)=K+\frac{c}{\gamma}\ln\frac{\eta\gamma}{c}+\frac{c}{\gamma}-cx=\eta-cx,$$
for $x<0$ when  $\delta\geq\frac{\mu^2+2r\sigma^2}{2\mu}$.
Thus the minimal cost function $V$ can be
 extended to $(-\infty,\infty)$, and it is continuous at $0$.
We should notice, however that $V'$ is not continuous at $0$.
\end{remark}
\subsection{Verification Theorem}
\label{Ver}
In this subsection, we prove the verification theorem.
\begin{theorem}
Suppose $W\in C^2[0,\infty)$ solves equation \eqref{E1}
with the boundary conditions \eqref{bd} and $W'$ is bounded 
on $[0,\infty)$. Then $W$ coincides with the value function
$V$, i.e. $W(x)=V(x)$ on $[0,\infty)$. Furthermore if there exists a constant
$u^*$ that minimizes the right side of \eqref{E1}, i.e.
$$\mathcal{L}^{u^*}V(x)-rV(x)=\underset{u\in\mathcal{U}}{\min}[\mathcal{L}^uV(x)-rV(x)],$$
for all $x>0$, then the optimal
reinsurance strategy is constant given by
$$u^*_s=u^*.$$
Optimal size of the cash injections is constant with the amount equal to
the minimizer $\xi^*$ of the inf-convolution operator, that is
$$g(\xi^*)+V(\xi^*)=\underset{\xi>0}{\inf}[g(\xi)+V(\xi)].$$
And the optimal injection times are the ruin times of the surplus process $X^{\pi^*}_s$.
\end{theorem}
\begin{proof}
Consider any admissible control $\pi$. Lemma~\ref{l3}, implies that
we need to be concerned with the policies in which the injections are made at the ruin times.
Put $\tau_0=0$ and denote by $\tau_i, i=1,2,...$ 
the ruin times of the surplus process under policy $\pi$. Denote by
$\xi_i,i=1,2,...$ the size of the $i$th injection.
Applying Ito's Lemma and taking into account that  $W$ solves \eqref{E1}, we have
\begin{equation}
\label{ineq1}
\begin{split}
&e^{-rt}W(X^\pi_t)-W(x)\\
=&\int_{0}^{t}
e^{-rs}[\mathcal{L}^{u_s}W(X^\pi_s)-rW(X^\pi_s)]ds\\
&+\int_{0}^{t}
\sigma(u_s)e^{-rs}W'(X^\pi_s)dw_s
+\underset{0<\tau_i\leq t}{\sum}e^{-r\tau_i}[W(X^\pi_{\tau_i})-W(X^\pi_{\tau_i-})]\\
\geq&\int_{0}^{t}
\sigma(u_s)e^{-rs}W'(X^\pi_s)dw_s
+\underset{0<\tau_i\leq t}{\sum}e^{-r\tau_i}[W(X^\pi_{\tau_i})-W(X^\pi_{\tau_i-})]\\
=&\int_{0}^{t}
\sigma(u_s)e^{-rs}W'(X^\pi_s)dw_s+\underset{0<\tau_i\leq t}{\sum}e^{-r\tau_i}[W(\xi_i)-W(0)],
\end{split}
\end{equation}
for any initial surplus level $x>0$. 
From the boundary condition $$MW(0)=W(0),$$ it holds
$$W(0)\leq g(\xi_i)+W(\xi_i).$$
Hence we get
\begin{equation}
\label{ineq2}
e^{-rt}W(X^\pi_t)-W(x)\geq \int_{0}^{t}
\sigma(u_s)e^{-rs}W'(X^\pi_s)dw_s+\underset{0<\tau_i\leq t}{\sum}e^{-r\tau_i}[-g(\xi_i)].
\end{equation}
Note that the quantity $\sigma(u_s)e^{-rs}W'(X^\pi_s)$ is bounded and thus the process
$$\{\int_{0}^{t}\sigma(u_s)e^{-rs}W'(X^\pi_s)dw_s\}_{t\geq 0}$$ is a zero mean martingale.
Taking expectation of both sides of \eqref{ineq2} and letting $t\rightarrow \infty$, we obtain
$$W(x)\leq E[\underset{0<\tau_i<\infty}{\sum}e^{-r\tau_i}g(\xi_i)]=C^\pi(x).$$
Thus 
$$W(x)\leq V(x).$$
When we set  $\pi=\pi^*$, the inequalities \eqref{ineq1} and \eqref{ineq2} become equalities.
As a result
 $$W(x)=C^{\pi^*}(x)\geq V(x).$$
Hence we conclude that $W$ coincides with $V$.
\end{proof}

\section{Optimal Solution in the Case of Excess-of-loss Reinsurance}
In this section, we seek for the optimal solution in the case when the insurance entity purchases
excess-of-loss reinsurance.

Suppose $V$ solves \eqref{E1}, or
\begin{equation}\label{E5}
\underset{0\leq u\leq N}{\min}\{\frac{1}{2}\sigma^2(u)V''(x)+[\mu(u)-\delta]V'(x)-rV(x)\}=0.
\end{equation}
By differentiation, the minimizer of $\mathcal{L}^uV(x)$ is
$$u_V(x):=-\frac{V'(x)}{V''(x)},$$
provided $0\leq u_V(x)\leq N$ and $V''(x)>0$.
Now we conjecture a solution given by
$$V(x)=\varsigma e^{-\kappa x},$$
with constants $\varsigma, \kappa>0$ to be determined. 
Thus we have $V'=-\kappa V$, $V''=\kappa^2 V$, and
$$u_V(x)=\frac{1}{\kappa}.$$
Furthermore  suppose $0\leq u_V(x)=1/\kappa<N$. Then  from \eqref{E5}, it follows that
 $\kappa$ is a root of 
\begin{equation*}
\frac{1}{2}\sigma^2(1/\kappa)\kappa^2-[\mu(1/\kappa)-\delta]\kappa-r=0.
\end{equation*}
Write 
\begin{equation}\label{J}
J(u):=\frac{\sigma^2(u)}{2u}-\mu(u)-ru+\delta,
\end{equation}
then $\kappa$ solves
\begin{equation}\label{kappa}
J(1/\kappa)=0.
\end{equation}
Note that $J(0+)=\delta$ and $J$ is a decreasing function. If $J(N)<0$, or equivalently,
\begin{eqnarray*}
\delta<\mu+rN-\frac{\sigma^2}{2N},
\end{eqnarray*}
then equation \eqref{kappa} has a unique solution $\kappa$ such that $0<1/\kappa<N$. 
From the boundary condition $V(0)=MV(0)$, one can determine $\varsigma$.
Similar to the previous section, we see that the  optimal cash injection is given by
\begin{equation}
\xi^*=\frac{1}{\kappa}\ln\frac{\varsigma\kappa}{c},
\end{equation}
and $\varsigma \in [c/\kappa,\infty)$ solves
\begin{equation}\label{varsigma}
K+\frac{c}{\kappa}\ln\frac{\varsigma\kappa}{c}+\frac{c}{\kappa}-\varsigma=0.
\end{equation}
Summarizing the above discussions we obtain:
\begin{theorem}\label{th3}
If $\delta<\mu+rN-\frac{\sigma^2}{2N}$, then the value function defined 
in \eqref{V} is given by 
$$V(x)=\varsigma e^{-\kappa x},$$
where $\kappa$ is a solution of  \eqref{kappa} and $\varsigma$ is the unique solution
of equation \eqref{varsigma} on the interval $(c/\kappa,\infty)$.
The optimal reinsurance control is equal to a constant given by
$$u^*_s\equiv \frac{1}{\kappa}.$$
The optimal cash injections $\xi_i^*$, $i=1,2,...$ are equal to a constant given by
$$\xi_i^*\equiv\frac{1}{\kappa}\ln\frac{\varsigma\kappa}{c}.$$
And the optimal times of injections $\tau_i^*$ are the ruin times of the surplus process under policy $\pi^*$.
\end{theorem}
\begin{theorem}\label{th4}
If $\delta\geq\mu+rN-\frac{\sigma^2}{2N}$, then the value function defined 
in \eqref{V} is identical to that in Theorem~\ref{th2}. Optimal reinsurance control
is $u^*=N$, i.e. no reinsurance purchase. 
The optimal times of injections are ruin times. And the optimal injection is identical to that in Theorem~\ref{th2}.
\end{theorem}
\begin{proof}
We need to check that $V$ solves equation~\eqref{E1}.
Note that $$u_V(x)=-\frac{V'(x)}{V''(x)}=\frac{1}{\gamma},$$ 
and the function $\mathcal{L}^uV$ decreases 
for $u\in (0,1/\gamma)$. Furthermore 
$\frac{1}{\gamma}\geq N$,
which is equivalent to $\delta\geq\mu+rN-\frac{\sigma^2}{2N}$.
Thus the minimizer of quantity $\mathcal{L}^uV(x)$  in the interval $[0,N]$ is $u^*=N$.
In addition  $V$ solves equation \eqref{E4}. Thus $V$ solves the equation~\eqref{E1}:
$$\underset{0\leq u\leq N}{\min}[\mathcal{L}^uV(x)-rV(x)]=\mathcal{L}^NV(x)-rV(x)=0,$$
with the boundary conditions \eqref{bd} (see definition of $\xi^*$ in \eqref{xi-star}).
The rest follows from  the verification theorem.
\end{proof}
\begin{remark}
In both cases of reinsurance, 
when the liability rate is high, the optimal policy takes zero
reinsurance. Consequently the value functions are identical.
\end{remark}
\section{Examples and Conclusions} 
In this section we give a few examples with explicit solutions and some concluding remarks. The first two examples are on proportional reinsurance and the other two are on
excess-of-loss reinsurance.
\begin{example}\label{e1} In this example, we consider a low debt liability $\delta$ with
proportional reinsurance. The parameters are:
$\mu=4$, $r=0.1$, $c=1.1$, $K=0.2$, $\sigma=0.8$ and $\delta=1.5$.
Then the optimal value function is $V(x)=0.5088e^{-8.4x}$, the
optimal constant reinsurance control is $u^*=0.7440$, and
the optimal constant cash injection is $\xi^*=0.1616$. See Theorem~\ref{th1} for detail.
\end{example}
\begin{example} In this example, we consider a high debt liability $\delta$ with
proportional reinsurance. We set $\delta=2.5$ and the other parameters the same as in Example \ref{e1}.
Then the optimal value function is $V(x)=0.6673e^{-4.9349x}$, the
optimal constant reinsurance control is $u^*=1$ (no reinsurance), and
the optimal constant cash injection is $\xi^*=0.2222$. For detailed calculations, see Theorem~\ref{th2}.
\end{example}
\begin{example}
In this example, we consider excess-of-loss reinsurance and assume
the individual claims follow an exponential distribution with distribution
function $F(x)=1-e^{-\theta x}$, where $\theta>0$. 
Note
\begin{equation}
\begin{split}
\mu(u)&=\frac{1}{\theta}(1-e^{-\theta u}),\\
\sigma^2(u)&=\frac{2}{\theta^2}(1-e^{-\theta u}-\theta u e^{-\theta u}),
\end{split}
\end{equation}
and assume $\theta=0.5$.
Set the other model parameters as the following:
$r=0.1$, $c=1.1$, $K=0.2$ and $\delta=1.5$.
Then the value function is given by $V(x)=5.6837e^{-0.2592x}$,
the optimal reinsurance control is $u^*=3.8580$, and the optimal
cash injection amount is $\xi^*=1.1271$. For detailed calculations, see Theorem~\ref{th3}.
\end{example}
\begin{example}
In this example, we consider the case when claims follow a Pareto distribution
with distribution function $F(x)=1-\left(\frac{b}{x+b}\right)^a$,
where $b>0$ and $a>2$. Short calculations give
\begin{equation}
\begin{split}
\mu(u)&=\frac{b^a}{1-a}[(u+b)^{1-a}-b^{1-a}],\\
\sigma^2(u)&=\frac{2b^a}{(1-a)(2-a)}[(2-a)u(u+b)^{1-a}-(u+b)^{2-a}+b^{2-a}].
\end{split}
\end{equation}
Set $a=3$ and $b=1$.
Assume the other model parameters are the same as in Example 3.
Then the value function is given by $V(x)=13.7587e^{-0.0958x}$,
the optimal reinsurance control is $u^*=10.3520$, and the optimal
cash injection amount is $\xi^*=1.8894$. See Theorem~\ref{th3} for detail.
\end{example}
 
In this paper, we studied an optimal injection-reinsurance control problem with a requirement of no ruin.
The related  QVIs were solved and optimal value functions were found to be exponential.
Optimal control policy was found. It was shown that the optimal reinsurance
level as well as the optimal amounts of the cash injections are constant.
The general controlled surplus process was modeled by a jump-diffusion process.
And under the optimal policy,
the surplus process became a Brownian motion 
with a constant drift and constant jumps at the times of ruin.
Future studies may consider problems with control components which include
investment, dividend payments, and various types of reinsurance.


\end{document}